\documentclass[12pt,letterpaper,titlepage]{amsart}
\usepackage{amsmath, amssymb, amsthm, amsfonts,amscd,xr}
\usepackage[all]{xy}
\CompileMatrices
\input epsf.tex
\newcommand{\N}{\mathbb{N}}

\newcommand{\R}{\mathbb{R}}
\newcommand{\C}{\mathbb{C}}

\def\beq{\begin{equation}}
\def\eeq{\end{equation}}

\def\arr{\hbox to 20pt{\rightarrowfill}}

\def\Gl{\mathrm {Gl}}
\def\Sl{\mathrm {Sl}} 
\def\SO{\mathrm {SO}} 
\def\A{\mathcal A} 
 
\def\cC{\mathcal C} 
\def\H{\mathcal H} 
\def\O{\mathcal O} 
 
\def\W{\mathcal W} 
\def\cR{\mathcal R}

\def\U{\mathcal U} 
\def\F{\mathbf F}

\newenvironment{res} 
               {\begin{equation} 
\begin{minipage}{0.85\textwidth}} 
               { \end{minipage}\end{equation} } 
\def\ber{\begin{res} } 
\def\eer{\end{res}} 
 
\numberwithin{equation}{section} 
\newtheorem{thm}{Theorem}[section]

\newcommand{\oline}{\overline}

\newtheorem{lemma}[thm]{Lemma} 
\newtheorem{lem}[thm]{Lemma} 
 
\newtheorem{prop}[thm]{Proposition}

\newtheorem{rem}[thm]{Remark}

\makeatletter 
\def\section{\@startsection {section}{1}{\z@}{3.5ex plus 1ex minus 
    .2ex}{2.3ex plus .2ex}{\large\bf}} 
    \def\subsection{\@startsection{subsection}{2}{\z@}{3.25ex plus 1ex minus 
 .2ex}{1.5ex plus .2ex}{\bf}} 
\makeatother 
\def\pf{{\em Proof}.\, } 
 
\def\bysame{\leavevmode\hbox to3em{\hrulefill}\,} 
 
\def\Ad{\operatorname{Ad}}

\def\af{\mathfrak{a}} 

\def\gf{\mathfrak{g}} 
\def\h{\mathfrak{h}}

\def\Hor{\operatorname{Hor}}
\def\kf{\mathfrak{k}} 
 
\def\m{\mathfrak{m}} 
\def\nf{\mathfrak{n}}

\def\pf{\mathfrak{p}} 
 
\def\qf{\mathfrak{q}}

\def\zf{\mathfrak{z}}
 
\def\A{\mathcal{A}}

\def\Ind{\operatorname{Ind}}
 
\def\O{\mathcal{O}}

\def\A{\mathcal{A}} 
 
\def\H{\mathcal{H}} 
\def\F{\mathcal{F}}

\def\S{\mathcal{S}} 
 
\def\W{\mathcal{W}}

\renewcommand{\Re}{\mbox{\rm Re}\,} 
 
\makeindex 
\begin{document} 
\title[Radon transform]
{Radon transform on real symmetric varieties: kernel and cokernel}

\author{Bernhard Kr\"otz} 
\address{Max-Planck-Institut f\"ur Mathematik\\  
Vivatsgasse 7\\ D-53111 Bonn
\\email: kroetz@mpim-bonn.mpg.de}

\date{\today} 
\thanks{}
\maketitle 

\section{Introduction}
Our concern is with 
$$Y=G/H$$
a semisimple irreducible real symmetric variety (space).\footnote{This 
means $G$ is a connected real semisimple Lie group, $H$ is the fixed point 
group of an involutive automorphism $\sigma$ of $G$ such that there 
is no $\sigma$-stable normal subgroup $H\subset L \subset G$ 
with $\dim H< \dim L < \dim G$.}
\par Our concern is with 
$$L^2(Y)$$
the space of square integrable function on $Y$ with respect to 
a $G$-invariant measure. This Hilbert space has a natural 
splitting 
$$L^2(Y)=L_{\rm mc}^2(Y)\oplus L_{\rm mc}^2(Y)^\bot$$
into most continuous part and its orthocomplement.
\par Another function space is of need, namely: 
$$\A:=L^1(Y)^\omega\, ,$$
the space of analytic vectors for the left regular 
representation of $G$ on $L^1(Y)$. Further we set 
$$\A_{\rm mc}:=\A\cap L_{\rm mc}^2(Y)\quad\hbox{and}\quad 
\A_{\rm mc}^\bot:=\A\cap L_{\rm mc}^2(Y)^\bot\, .$$
We believe that $\A_{\rm mc}$ is dense in 
$L^2(Y)_{\rm mc}$ and 
that  $\A_{\rm mc}^\bot$ is dense in the subspace of $L^2(Y)_{\rm mc}^\bot$ 
which corresponds to principal series which are induced from 
integrable representations of their Levi subgroups -- a proof of 
this and similar facts for $\A$ replaced by  other function 
spaces is desirable.

\par Our concern is with an open domain in parameter space of generic 
real horospheres 
$$\Xi=G/(M\cap H) N$$
where $MAN$ is a minimal $\sigma\theta$-stable\footnote{$\theta$ is 
a Cartan involution commuting with $\sigma$.}
parabolic subgroup of $G$. 

\par Write $C_0^\omega(\Xi)$ for the space of analytic  
functions on $\Xi$ which vanish at infinity. 
In this paper we verify the following facts: 
\begin{itemize}
\item The map 
$$\cR: \A\to C_0^\omega(\Xi), f\mapsto \left(gM_H N\mapsto 
\int_N f(gnH)\ dn \right)$$
is well defined. (We call $\cR$ the (minimal) Radon transform)
\item $\cR|_{\A_{\rm mc}^\bot}=0$. 
\item $\cR|_{\A_{\rm mc}\cap \S(Y)}$ is injective\footnote{$\S(Y)$ is the 
Schwartz space of rapidly decaying functions}. 
\end{itemize}

\par {\it Acknowledgement:} The above results are motivated 
by discussions with Simon Gindikin during my stay at IAS in 
October 2006. I am happy to acknowledge his input and the hospitality 
of IAS.  
\par I would like to thank Henrik Schlichtkrull for 
pointing out a mistake and several inaccuracies 
in an earlier version of the paper. 
Also I thank the anonymous referee for pointing out a mistake 
and his many very useful requests on more detail.

\section{Real symmetric varieties}

\subsection{Notation}

The objective of this section is to 
introduce notation and to recall some facts 
regarding real symmetric varieties. 

\par Let $G_\C$ be a simply connected linear 
algebraic group whose Lie algebra $\gf_\C$ we assume 
to be semi-simple. We fix a real form $G$ of $G_\C$: this means 
that $G$ is the fixed point set of an involutive automorphism 
$\sigma$ of $G_\C$ and that $\gf$, the Lie algebra of $G$, 
yields $\gf_\C$ after complexifying. 

\par Let now $\tau$ be a second involutive 
automorphism of $G_\C$ which we request to commute with $\sigma$. 
In particular, $\tau$ stabilizes $G$. We write 

$$H_\C:= G_\C^\tau\quad\hbox{and} \quad H:= G^\tau$$
for the corresponding fixed point groups of $\tau$ in 
$G$, resp. $G_\C$. We note that $H_\C$ is always connected, 
but $H$ usually is not; the basic example of $(G_\C,G)=(\Sl(2,\C), 
\Sl(2,\R))$ and $(H_\C,H)=(\SO(1,1;\C), \SO(1,1;\R))$ already illustrates 
the situation. 
\par With $G$ and $H$ we form the object of our concern 

$$Y=G/H\, ;$$ 
we refer to $Y$ as a {\it real (semi-simple) symmetric variety (or space)}. 
Henceforth we will denote by $y_o=H$ the standard 
base point in $Y$. 
We write $Y_\C=G_\C/H_\C$ for the affine complexification of $Y$ and 
view, whenever convenient, $Y$ as a subspace of $Y_\C$ via
the embedding 

$$Y\hookrightarrow Y_\C, \ \ gH\mapsto gH_\C\, .$$

\par At this point it  is useful to introduce infinitesimal notation. 
Lie groups will always be denoted by upper case Latin letters, 
e.g. $G$, $H$, $K$ etc., and the corresponding Lie algebras 
by lower case German letters, eg. $\gf$, $\h$, $\kf$ etc. 
It is convenient to use the same symbol $\tau$ for the 
derived automorphism $d\tau({\bf 1})$ of $\gf$. 
Let us denote by $\qf$ the $-1$-eigenspace of $\tau$ on 
$\gf$. Note that $\qf$ is an $H$-module which naturally 
identifies with the tangent space $T_{y_o} Y$ at the 
base point.

\par From now we will request that $Y$ is irreducible, i.e. 
we assume that the only $\tau$-invariant ideals in $\gf$ are 
$\{ 0\}$ and $\gf$. In practice this means that $G$ is simple 
except for the group case $G/H=H\times H/ H \simeq H$. 

\par We recall that maximal compact subgroups $K<G$ are in one-to-one 
correspondences with Cartan involutions $\theta: G\to G$. The correspondence 
is given by $K=G^\theta$. We form the Riemann symmetric space 
$$X=G/K$$
of the non-compact type and denote by $x_o=K$ the standard base point. 
As before we write $\theta$ for the derived involution 
on $\gf$. We let $\pf\subset \gf$ be the $-1$-eigenspace 
$\theta$ and note that the $K$-module $\pf$ identifies 
with $T_{x_o} X$.

\par According to Berger,  we may (and will) assume  
that $K$ is $\tau$-invariant.  This implies that 
both $\h$ and $\qf$ are $\tau$-stable. Let us 
fix a maximal abelian subspace 

$$\af \subset \qf\cap \pf\, .$$
We wish to point out that  $\af$ is unique modulo conjugation 
by $H\cap K$, see \cite{S}, Lemma 7.1.5. Set $A=\exp (\af)$.

\par Our next concern is the centralizer $Z_G(A)$ of $A$. 
We first remark that $Z_G(A)$ is reductive and admits a natural splitting 
$$ Z_G(A)= A\times M\, , $$
(cf. \cite{K}, Prop. 7.82 (a)). The Lie algebra of $M$ is given by 

$$\m=\zf_\gf(\af) \cap \af^\perp$$
where $\af^\perp$ is the orthogonal complement of 
$\af$ in $\gf$ with respect to the Cartan-Killing form $\kappa$ of $\gf$. 
If $M_0$ denotes the connected component of $M$, then 
$$M=M_0 F$$
where $F\subset M\cap K$ is a finite 2-group 
(cf.\ \cite{K}, Prop. 7.82 (d) and Th. 7.52). 

\begin{rem} If $\af$ is maximal abelian in $\pf$, then $F\subset H$ 
as follows from the explicit description of $F$ in \cite{K}, Th. 7.52. 
In general however, the $\tau$-stable group $F$ is not contained in $H$ and does 
not  even admit a factorization $F=F^\tau F^{-\tau}$ in 
$\tau$-fixed and $\tau$-anti-fixed points.
\end{rem}

We write $\m_{ns}$ for the non-compact semisimple part 
of $\m$ and note that 

\begin{equation}\label{minh}
\m_{ns}\subset \h\end{equation} 
(cf. \cite{S}, Lemma 7.1.4).  
Set $M_H=M\cap H = Z_H(A)$ and let $\m = \m_h +\m_q$ be the splitting 
of $\m$ into $+1$ and $-1$-eigenspace. Note 
that $\m_h$ is the Lie algebra of $M_H$. 
Then (\ref{minh}) implies that $\m_q\subset \kf$ 
and consequently $M_q=\exp(\m_q)$ 
is compact. Moreover: 

$$M/F= M_H M_q/F  \quad \hbox{with}\quad   
M_H\cap M_q \quad  \hbox{discrete}\, .$$

\par We turn our attention to the root space decomposition 
of $\gf$ with respect to $\af$. 
For $\alpha\in \af^*$ , let
$$\gf^\alpha=\{ X\in \gf\mid (\forall Y\in \af)\ [Y,X]=\alpha(Y) X\} $$
and set  
$$\Sigma=\{ \alpha\in \af^*\setminus\{ 0\}\mid \gf^\alpha\neq \{0\}\}\,  .$$  
It is a fact that $\Sigma$ is a (possibly reduced) 
root system, cf. \cite{S}, Prop. 7.2.1. 
Hence we may fix a positive system $\Sigma^+\subset \Sigma$ and define 
a corresponding nilpotent subalgebra 

$$\nf:=\bigoplus_{\alpha\in \Sigma^+} \gf^\alpha\, .$$
Set $N:=\exp(\nf)$. 
Note that $\tau(\nf)=\theta(\nf)$. We record the decomposition 

$$\gf=\af \oplus \m \oplus \nf \oplus \tau(\nf)\, .$$

\par We shift our focus to the real flag manifold of $G$ associated 
to $A$ and $\Sigma^+$. We define 
$$P_{\rm min} := M A N$$ 
and note that $P_{\rm min}$ is a minimal $\theta\tau$-stable parabolic 
subgroup of $G$. 

\par The open $H$-orbit decomposition on the flag manifold $G/P_{\rm min}$ 
is essential  in the theory of $H$-spherical 
representations of $G$. In order to describe this decomposition 
we have to collect some facts on Weyl groups first. 

\par Let us denote by $\W$ the Weyl group of the root system $\Sigma$. The 
Weyl group admits an analytic realization: 

$$\W= N_K(\af)/ Z_K(\af)\, .$$
The group $\W$ features a natural subgroup 

$$\W_H:= N_{H\cap K}(\af)/ Z_{H\cap K}(\af)\, . $$
Knowing $\W$ and $\W_H$, we can quote the
decomposition of $G$ into open $H\times P_{\rm min}$-cosets (cf. \cite{Ma}): 

\begin{equation}\label{oo}
G\doteq  \amalg_{w\in \W_H \setminus \W} H w P_{\rm min}\, , 
\end{equation}
where $\doteq$ means equality up to a finite union of strictly lower 
dimensional 
$H\times P_{\rm min}$-orbits. 

\subsection{Horospheres}

This paragraph is devoted to horospheres on the 
symmetric variety $Y$. By a {\it (generic) horosphere} on $Y$ we 
understand an orbit of a conjugate of $N$ of maximal 
dimension (i.e. $\dim N$). The entity of all 
horospheres will be denoted by $\Hor (Y)$. We 
remark that $G$ acts naturally on $\Hor (Y)$ from the left. 

\par Our goal is to show that $\Hor(Y)$ is a connected 
analytic manifold. For that we define 

$$G_h:=\{ x\in G\mid Nx\cdot y_o\in \Hor(Y)\}$$
and note the following immediate things: 
\begin{itemize}
\item $G_h$ is open, right $H$-invariant and left $P_{\rm min}$-invariant. 
\item $G_h$ contains the open $P_{\rm min}\times H$-cosets
$P_{\rm min}w H$  where $w\in \W/\W_H$, see (\ref{oo}). In particular, 
$G_h$ is dense.   
\item $\Hor(Y)=\{ gNx\cdot y_o\mid g\in G, x\in G_h\} .$
\item (Infinitesimal characterization) 
$G_h=\{ x\in G\mid \Ad(x^{-1}) \nf \cap\h =\{0\}\}\, .$
\end{itemize}

\begin{rem} The set $G_h$ is in general bigger then the open dense 
disjoint union $\bigcup_{w\in \W/ \W_H } P_{\rm min} w H$. It is 
in particular connected as we show below. To see an example, 
consider $G=\Sl(2,\R)$, $H=\SO(1,1;\R)$ and $P_{\rm min}$ the 
upper triangular matrices with determinant one. Then 
$\h$ and $\nf$, both one-dimensional, can never be conjugate. Thus, 
by the infinitesimal characterization  from above, one has $G_h=G$
in this case.
\end{rem}

\par Next we provide charts for $\Hor(Y)$. 
For that we introduce the 
$G$-manifold 
$$\Xi=G/ M_HN\, $$
and define the $G$-equivariant map 

$$E : \Xi\to \Hor (Y), \ \ \xi=gM_H N \mapsto 
E(\xi)=gN\cdot y_o\, .$$
As in \cite{GKO2}, Prop. 2.1, one verifies that $E$ is 
an injection. 
Now we move our chart by elements $x\in G_h$. Set 
$L:= MA$, $L^x= x^{-1}Lx$, $L_H^x:= L^x\cap H$, 
$N^x:=x^{-1} N x$  
and 
$\Xi_x:= G/ L_H^x N^x$. Then the 
map 

$$E_x:G/L_H^x N^x \to \Hor (Y), \ \ gL^x N^x \mapsto g N^x\cdot y_o\, .$$ 
is $G$-equivariant and injective. It is 
immediate that $(E_x, \Xi_x)_{x\in G_h}$ forms an analytic atlas 
for $\Hor(Y)$. 

\begin{lem} $\Hor(Y)$ is connected. 
\end{lem}

\begin{proof} It is sufficient to show that $G_h$ is connected. 
We know that $G_h$ is an open and dense 
$P_{\rm min}\times H$-invariant subset of $G$. Now there are only 
finitely many orbits of $P_{\rm min}\times H$ on $G$ and those
are described explicitly, see \cite{Ma}. 
\par As $G_h$ contains all 
open orbits, it is sufficient to show that $G_h$
contains all codimension one orbits. This in turn follows from the
explicit description of all orbits in \cite{Ma}, Th. 3 (i): if 
$HcP_{\rm min}\subset G$ is not open, then \cite{Ma} implies 
that $\Ad(c) \af\cap \h\not =\{0\}$. In 
particular if $HcP_{\rm min}$ is of codimension one, then 
$\Ad(c) (\m +\af +\nf)\cap \h= \Ad(c) \af\cap \h$ and therefore 
$\Ad(c)\nf \cap \h =\{0\}$.  Our infinitesimal characterization 
completes the proof. 
\end{proof}

Finally we discuss polar coordinates on $\Xi$.

\begin{equation}\label{par}  K/(M_H\cap K) \times A \to \Xi, \ \ (k(M_H\cap K), a)\mapsto kaM_H N
\end{equation}
is a diffeomorphism. Often we view $A$ as subspace of $\Xi$ via 
$$A\hookrightarrow \Xi, \ \ a\mapsto a M_H N\, .$$

\section{Function spaces and the definition of the Radon transform}

We consider the left regular representation $L$ of 
$G$ on $L^1(Y)$, i.e. for $g\in G$ and $f\in L^1(Y)$ 
we look at 
$$[L(g)f](y)= f(g^{-1}y) \qquad (y\in Y)\, .$$ 
Then we focus on the subspace 

$$\A:=L^1(Y)^\omega$$
of analytic vectors for $L$. We note that 
$f\in\A$ means that $f\in C^\omega(Y)$ such that 
there exists an open neighborhood $U$ of $\bf 1$ in $G_\C$ such that 
$f$ extends holomorphically to 
$$UG\cdot y_o \subset Y_\C=G_\C/H_\C$$
(here we view $Y$ embedded in $Y_\C$ via $gH\mapsto gH_\C$)
such that for all compact $C\subset U$: 

\begin{equation} \label{e=c} \sup_{c\in C} \|f(c\cdot)\|_{L^1(Y)}<\infty\,, 
\end{equation}
(cf. \cite{gko}, Prop.\ A.2.1). 
For an open neighborhood $U$ of $\bf 1$ in $G_\C$ we denote 
by $\A_U$ the space of holomorphic functions on 
$UG\cdot y_o$ which satisfy (\ref{e=c}) for all compact $C\subset U$. 
Note that $\A_U$ can be seen as a closed subspace of 
$\O(U, L^1(Y))$ and hence is a Fr\'echet space. 
Moreover 
$$\A=\bigcup_U \A_U$$
with continuous inclusions 
$$\A_U\to \A_V, \ \ f\mapsto f|_{VG\cdot y_o}$$
for $V\subset U$. In this way we can endow $\A$ with a structure of 
of a locally convex space. 

\par We observe that $G$, via $L$,  acts on $\A$ in an analytic 
manner. 

\subsection{Definition of the Radon transform}\label{sd}

We begin with the crucial technical fact. 

\begin{lemma} \label{l=1}Let $f\in\A$. Then the following assertions hold: 
\begin{enumerate}
\item $\sup_{a\in A\atop k\in K} \int_N |f(kan\cdot y_o)| \ dn <\infty$. 
\item $\sup_{a\in A\atop k\in K} 
a^{-2\rho}\int_N |f(kna\cdot y_o)| \ dn <\infty$. 
\end{enumerate}
\end{lemma}

\begin{proof} (i) Let $f\in\A$. Let 
$B_\af\subset \af$ and $B_\nf\subset\nf$ be balls around zero 
and set 
$$U_A:= \exp(B_\af + iB_\af)\subset A_\C\,  ,$$ 
$$U_N:= \exp(iB_\nf)\exp(B_\nf)\subset N_\C\, .$$
If we choose $B_\af$ and $B_\nf$ small enough, then 
$f$ will extend to a holomorphic function in a 
neighborhood of 
$$KU_N U_A G\cdot y_o$$ 
such that 
$$\sup_{c\in KU_N U_A} \|f(c\cdot )\|_{L^1(Y)}<\infty\, .$$
To reduce notation let us assume that $M\subset H$ -- this is no loss
as the complementary piece $M_qF$ in $M$ to $M_H$ is compact, 
(\ref{minh}).  
Then 
$$NA\to Y , \ \ na\mapsto na\cdot y_o$$ 
is an open embedding. In particular $NA\cdot y_o\subset Y$ is open. 
It follows that 

$$U_A U_N AN \cdot y_o \subset Y_\C$$
is open and we may assume that 
$$U_AU_N AN\subset A_\C N_\C$$ injects 
into $Y_\C$.  Fix $k\in K$ and  define a holomorphic function $F$ on 
$U_N U_A AN$ by $F(z):=f(kz \cdot y_o)$. 
In particular, we obtain a constant $C>0$ such that for 
all $an\in AN$ we get that 

$$|F(an)|\leq C \int_{U_N U_A } |F(n'a' an)| da' dn'$$
with $da'$ and $dn'$ Haar-measures on $A_\C$ and $N_\C$ 
(Bergman-estimate). 
Let us write $dy$ for a Haar measure on $Y$ and observe 
that $dy$ restricted to $AN$ is just $da \ dn$ with $da$ and $dn$ 
Haar measures on $A$ and $N$ respectively.   
Therefore 
\begin{align*} \int_N |F(an)| dn &\leq C  \int_{U_N U_A } 
\int_N |F(n'a' an)|\  da'\  dn' \ dn \\ 
&\leq C  \int_{U_N} 
\int_{B_\af} \int_A \int_N |F(n'\exp(iX) a' an)| \ dn'\ dX  \ da'\  dn \\ 
&\leq C  \int_{U_N} 
\int_{B_\af} \int_Y  |f(kn'\exp(iX) y )| \ dn'\ dX  \ dy \\ 
& \leq C\cdot  {\rm{vol}} (U_N)\cdot {\rm{vol}} (B_\af) \sup_{c\in KU_A U_N}
\|f(c\cdot)\|_{L^1(Y)}< \infty\, .
\end{align*}
We observe that the last expression does not depend on $a\in A$ and $k\in K$. 
This proves (i). 
Now (ii) is just a variable change of (i): 

\begin{align*} \int_N |f(kna\cdot y_o)| \ dn & 
= \int_N   |f(k a a^{-1} na\cdot y_o)| \ dn \\ 
&= a^{2\rho} \int_N   |f(k a n\cdot y_o)| \ dn 
\end{align*} 

\end{proof}

For $f\in \A$ we define a function $\cR(f)$ on $\Xi$ via 

$$\cR(f)(gM_H N):=\int_N f(gn \cdot y_o) \ dn \, .$$
According to our previous lemma the defining integrals 
are absolutely convergent. 

Let us write $C_0^\omega(\Xi)$ for the 
space of analytic functions on $\Xi$ which vanish at 
infinity. In view of (\ref{par}) vanishing at infinity for $F$ means 

$$\lim_{a\to \infty\atop a\in A} \sup_{k\in K} |F(kaM_H N)|=0\, .$$ 

\begin{prop}\label{p=d} The following assertions hold: 
\begin{enumerate}
\item For all $f\in \A$ one has $\cR(f)\in C_0^\omega(\Xi)$. 
\item The map $\cR: \A\to L^1(\Xi)^\omega\subset C_0^\omega (\Xi)$ is continuous. 
\end{enumerate}
\end{prop}

\begin{proof} (i) As before it is no loss to assume that $M\subset H$ -- the piece of $M$ not in 
$H$ is compact by (\ref{minh}). 
Let us first show that $F:=\cR(f)|_A\in C_0^\omega$. In fact 
as $f$ is in $L^1(Y)$ and $dy|_{AN} = da\ dn$,  it follows that $F\in L^1(A)$. 
Moreover $F\in L^1(A)^\omega$, i.e. it is an analytic vector for 
the regular representation of $A$ on $L^1(A)$. Therefore the 
standard Sobolev lemma implies that $F\in C_0^\omega(A)$. 
\par Finally, employing the additional compact parameter $k\in K$ 
causes no difficulty.    
\par (ii) This follows from (i) and the last (and crucial) estimate in the proof of 
Lemma \ref{l=1}(i). 
\end{proof}

\begin{rem}\label{ran} Actually one can define the Radon transform 
with image on the whole horosphere space $\Hor(Y)$. Recall 
the set $G_h\subset G$ and for $x\in G_h$ the parameter space 
$\Xi_x=G/L_H^x N^x$. For $f\in \A$ one then defines 
$$\cR_x(f)(gL_H^x N^x)=\int_N f(g x^{-1} n x\cdot y_o)\ dn\, . $$
The resulting function $\cR_x(f)$ is then, as above seen 
to lie in $C_0^\omega(\Xi_x)\cap L^1(\Xi_x)$. Patching matters together 
we thus obtain a well defined $G$-map 
$$\cR: \A\to C^\omega(\Hor(Y))\, .$$
\end{rem}

\section{The kernel of the Radon transform: discrete spectrum}

In this section we show that the discrete spectrum  of $L^2(Y)$, as far as it meets 
$\A$, lies in the 
kernel of $\cR$. In fact we show even more: namely that the trace of $\A$ in the orthocomplement 
of the most continuous spectrum lies in the kernel. 

\par Recall our minimal $\theta\tau$-stable parabolic subgroup 

$$P_{\rm min}= M AN\, .$$
In the sequel we use the symbol $Q$ for a $\theta\tau$-stable 
parabolic which contains $P_{\rm min}$. There are only finitely many. 
We write 
$$Q= M_Q A_Q N_Q$$ 
for its standard factorization and observe: 
\begin{itemize}
\item $M_Q\supset M$, 
\item $A_Q\subset A$, 
\item $N_Q\subset N$. 
\end{itemize}
One calls two parabolics $Q$ and $Q'$ associated if there 
exists an $n\in N_K(A)$  such that $n A_Q n^{-1}= A_{Q'}$.
This induces an equivalence relation $\sim$  on our parabolics and we 
write $[Q]$ for the corresponding equivalence class. If the 
context is clear we simply omit the brackets and just write 
$Q$ instead of $[Q]$.   

\par It follows from the Plancherel theorem (\cite{BS2}, \cite{D}) that 

$$L^2(Y)=\bigoplus_{Q\supset P_{\rm min}/\sim} L^2(Y)_{[Q]}$$
where $L^2(Y)_Q=L^2(Y)_{[Q]}$ stands for the part corresponding 
to representations which are induced off from $Q$ by discrete 
series of $M_Q/ M_Q \cap wHw^{-1}$ with $w$ running over 
representatives of $\W/ \W_H$.  

\par As $\A\subset C_0^\omega(Y)$ , see \cite{KS}, 
we observe that $\A\subset L^2(Y)$. However we note that 
$\A$ might not be dense in $L^2(Y)$: it has no components in this 
part of $L^2(Y)_Q$ which is induced from non-integrable 
discrete series of $M_Q/M_Q\cap wHw^{-1}$.

\par Let $\A_Q= L^2(Y)_Q\cap \A$. 
For the extreme choices of $Q$ we use a special terminology: 

$$L^2(Y)_{\rm disc}:= L^2(Y)_G \quad\hbox{and}\quad 
 L^2(Y)_{\rm mc}:=L^2(Y)_{P_{\rm min}}$$
and one refers to the {\it discrete} and {\it most continuous} part 
of the square-integrable spectrum. Likewise we declare 
$\A_{\rm disc}$ and $\A_{\rm mc}$.  Let us mention 
that we believe that 
$$\A_{\rm mc}\subset L^2(Y)_{\rm mc}\quad \hbox{is dense}\, $$
(the heuristic reason for that is that $M/M\cap H$ is compact).

\begin{thm}\label{Rd} $\cR(\A_{\rm disc})=\{0\}$.
\end{thm}

\begin{proof} The proof is the same as for the group, see 
\cite{W}, Th. 7.2.2 for a useful exposition. 
\par Let $f\in \A_{\rm disc}$. We have to show that 
$\cR(f)=0$. As $\cR$ is continuous (Proposition \ref{p=d}), 
standard density arguments reduce to the case 
where $f$ belongs to a single discrete series representation and that 
$f$ is $K$-finite. Let 
$$V=\U(\gf_\C)f$$
be the corresponding Harish-Chandra module and set 
$T:=[\cR|_V]|_A$. Then $T$ factors over the Jacquet module 
$j(V) =V/\nf V$. We recall that $j(V)$ is an
 admissible finitely generated 
$(M,\af)$-module. Hence 
$$\dim \U(\af) T(f)<\infty\, .$$
Consequently
$$T(f)(a)= \sum_\mu a^\mu p_\mu(\log a)\qquad (a\in A)$$
where $\mu$ runs over a finite subset in $\af_\C^*$  and 
$p_\mu$ is a polynomial
(see \cite{W}, 8.A.2.10). From 
$T(f)\in C_0^\omega(A)$,  we thus conclude that $T(f)=0$ and hence 
$\cR(f)=0$ by the $K$-finiteness of $f$. 
\end{proof}

As a consequence of the previous theorem we obtain 
the main result of this subsection.

\begin{thm}\label{Rd-2} Let $Q\supsetneq P_{\rm min}$. 
Then $\cR(\A_Q)=\{0\}$. 
\end{thm}

\begin{proof} If $Q=G$, then this part of the 
previous theorem. The general case will 
be reduced to that. So suppose that $P_{\rm min}\subsetneq
Q \subsetneq G$. Define $\Xi_Q= G/ (M_Q\cap H) N_Q$ and like in 
(\ref{par}) one has a diffeomorphic parameterization 
$[K\times_{M_Q\cap K} M_Q/ M_Q\cap H]\times A \to \Xi_Q$. 
\par As in Subsection \ref{sd}, one obtains 
that the map
$$\cR_Q:  \A \to L^1(\Xi_Q)^\omega, \ \ f\mapsto
\left(g(M_Q\cap H)N_Q\mapsto \int_{N_Q} f(gnH) \ dn\right)$$
is defined, $G$-equivariant and continuous. 
\par Next observe that 
$$N=N_Q \rtimes N^Q$$
with $\{{\bf 1}\} \neq N^Q \subset M_Q$. As before the  map
$$\cR^Q: L^1(\Xi_Q)^\omega\to L^1(\Xi)^\omega$$
$$f\mapsto 
\left(gM_HN \mapsto \int_{N^Q} f(gn (M_Q\cap H)N_Q) \ dn\right)$$
is defined, equivariant and continuous.  
\par Now we note that 

\begin{equation} \label{rf} \cR =\cR^Q \circ  \cR_Q\, .
\end{equation} 
\par Let now $f\in \A_Q$. Without loss of generality we may 
assume that $f$ belongs to a wave packet induced from 
a discrete series $\sigma \subset L^2(M_Q/ M_Q\cap H)$. 

Note that $M_Q/M_Q\cap H$ naturally embeds into 
$\Xi_Q$ and that the restriction of $L^1(\Xi_Q)^\omega$ to 
$M_Q/M_Q\cap H$ stays integrable. 
Hence $F:=\cR_Q(f)$ restricted to $M_Q/M_Q\cap H$
is integrable as well. 

\par We claim that $F|_{M_Q/M_Q\cap H}$
belongs to the $\sigma$-isotypical class. First note that 
$L^1(\Xi_Q)^\omega\subset L^2(\Xi_Q)$. By 
induction on stages 
$$L^2(\Xi_Q)=\Ind_{(M_Q\cap H)N_Q}^G {\rm triv}
\simeq \Ind_{M_Q N_Q}^G L^2(M_Q/M_Q\cap H)\, .$$
Thus if $L^2(M_Q/ M_Q\cap H)=\int_{\widehat M_Q}^\oplus m_\pi \H_\pi \ d\mu(\pi)$
is the Plancherel decomposition, then as $G$-modules:
$$L^2(\Xi_Q)\simeq 
\int_{\hat M_Q}^\oplus m_\pi \Ind_{M_Q N_Q}^G \H_\pi\ d\mu(\pi)\, .$$ 
Now note that $A_Q$ acts one right on $\Xi_Q$ and this 
action commutes with $G$ (see the next section for a detailed 
discussion for $Q=P_{\rm min}$). This gives us a further 
disintegration of the left regular representation $L_Q$ of $G$ 
on  $L^2(\Xi_Q)$: 

$$L_Q\simeq \int_{\widehat M_Q}^\oplus m_\pi \int_{i\af_Q^*}^\oplus
\Ind_{M_Q A_Q N_Q}^G [\pi\otimes (-\lambda- \rho_Q)\otimes{\bf 1}] \ d\lambda
\ d\mu(\pi)\, .$$ 
As $R_Q$ is $G$-equivariant, we thus conclude that 
$R_Q(f)\in \Ind_{M_Q N_Q}^G \sigma$. 

\par Our claim combined with the previous theorem implies that 
$$\cR^Q(F)|_{M_Q/M_H N^Q}=0\, .$$ 
By the equivariance properties of $\cR_Q$ and $\cR^Q$
we are free to replace 
$f$ (and hence $F$) by any $G$-translate. Consequently 
$\cR^Q(F)=0$, as was to be shown.
\end{proof}

\section{Restriction of the Radon transform to the most continuous 
spectrum}

The objective of this section is to show that $\cR$ is faithful 
on the most continuous spectrum. 

\par We recall a few 
facts on the spectrum of $L^2(\Xi)$ and the 
most continuous spectrum on $Y$ and start 
with the "horocyclic picture". The homogeneous 
space $\Xi$ carries a $G$-invariant measure. Consequently 
left shifts by $G$ in the argument of a function on $\Xi$ 
yields a unitary representation, say $L$, of $G$ on $L^2(\Xi)$; in symbols

$$(L(g)f)(\xi)= f(g^{-1}\cdot \xi)\qquad (f\in L^2(\Xi), 
g\in G, \xi\in \Xi)\, .$$ 
It is important to note that the $G$-action on $\Xi$ admits 
a commutating action of $A$ from the right
$$\xi\cdot a= ga MN \qquad (\xi=gM_HN\in\Xi, a\in A); $$
this is because $A$ normalizes $M_H N$. 
Therefore the description 

$$(R(a)f)(\xi)=  a^\rho \cdot f(\xi\cdot a) 
\qquad (f\in L^2(\Xi), a\in A, \xi\in \Xi)$$
defines a unitary representation $(R, L^2(\Xi))$  of $A$ 
which commutes with the $G$-representation $L$. 
Accordingly we define an $A$-Fourier transform 
for an appropriate function $f$ on $\Xi$ by 
$$\F_A(f)(\lambda, gM_HN)
:=\int_A  [R(a)f](gM_H N) a^{\lambda} \ da \qquad (\lambda\in i\af^*)\, .$$

\par For $\lambda\in\af_\C^*$ let us set 

\begin{align*}L^2(\Xi)_\lambda:=\{ f: G\to \C\mid & \bullet  
\  f \ \hbox{measurable}, \\ 
& \bullet \ f(\cdot man)= a^{-\rho -\lambda} f(\cdot)\ \forall man\in 
M_H A N,\\ 
&\bullet\ \int_K |f(k)|^2 \ dk <\infty\}\end{align*}

Likewise we write $C^\infty(\Xi)_\lambda$ for the 
smooth elements of $L^2(\Xi)_\lambda$. 
The disintegration of $L^2(\Xi)$ is then given by 

$$L^2(\Xi)\simeq \int_{i\af^*}^\oplus L^2(\Xi)_\lambda\ d\lambda $$
with isomorphism given by the $A$-Fourier transform 
$$f\mapsto (\lambda\mapsto \F_A (f)(\lambda, \cdot))\, .$$

\par In the next step we recall the Plancherel decomposition 
for the most continuous spectrum (cf. \cite{BS}). 

\par Some generalities upfront. For a representation $\pi$ of a 
group $L$ on some topological vector space $V$ we denote by 
$\pi^*$ the dual representation on the (strong) 
topological dual $V^*$ of $V$. 
  
\par Let $\sigma\in \widehat{M/M_H}$ and $V_\sigma$ a unitary
representation module for $\sigma$. 

For $\lambda\in \af_\C^*$ we define

 \begin{align*}\H_{\sigma,\lambda}:=\{ f: G\to V_\sigma\mid & \bullet  
\  f \ \hbox{measurable}, \\ 
& \bullet \ f(\cdot man)= a^{-\rho -\lambda} \sigma(m)^{-1}f(\cdot)\ \forall man\in P_{\rm min},\\ 
&\bullet\ \int_K \langle f(k), f(k)\rangle_\sigma  \ dk <\infty\}\, .\end{align*}   
The group $G$ acts on $\H_{\sigma,\lambda}$ by displacements from the left and the so-obtained Hilbert representation will be denoted 
by $\pi_{\sigma, \lambda}$.

\begin{rem}\label{rem=i} The relationship between $\H_{\sigma,\lambda}$ and 
$L^2(\Xi)_\lambda$ is as follows. If $\mu_\sigma$ is an (up to scalar
unique)  $M_H$-fixed 
element in $V_\sigma^*$, then the mapping 
$$\H_{\sigma,\lambda} \to L^2(\Xi)_\lambda, \ \ f\mapsto \mu_\sigma(f)$$
is a $G$-equivariant injection. The map can be made isometric by 
appropriate scaling of $\mu_\sigma$. Employing induction 
in stages one therefore obtains 
an isometric identification 
$$\widehat \bigoplus_{\sigma\in \widehat{M/ M_H}} \H_{\sigma, \lambda}
=L^2(\Xi)_\lambda\, .$$   
\end{rem}

\par Sometimes it is useful to realize $\H_{\sigma,\lambda}$ as 
$V_\sigma$-valued functions on $\oline N:=\theta(N)$; we speak of the non-compact 
realization then. Define a weight function on $\oline N$ by 

$$w_\lambda(\oline n)= a^{2 \operatorname{Re} \lambda}$$
where $a\in A$ is determined by $\oline n\in Ka N$. Then the map 

$$\H_{\sigma,\lambda} \to 
L^2 (\oline N, w_\lambda(\oline n) d\oline n)\otimes V_\sigma, \ \ f\mapsto f|_{\oline N}$$
is an isometric isomorphism.

We remark that: 
\begin{itemize}
\item $\pi_{\sigma, \lambda}$ is irreducible for generic $\lambda$. 
\item $\pi_{\sigma, \lambda}$ is unitary for $\lambda\in i\af^*$.
\item The dual representation of $\pi_{\sigma,\lambda}$ 
is canonically isomorphic to $\pi_{\sigma^*, -\lambda}$; the dual 
pairing is given by 
$$\langle f,g\rangle :=\int_{\oline N} (f(\oline n), g(\oline n))_\sigma \ d\oline n$$ 
for $f\in \H_{\sigma,\lambda}$, $g\in \H_{\sigma^*, -\lambda}$ and 
$(,)_\sigma$  the natural pairing between $V_\sigma$ and 
$V_\sigma^*$. 
\end{itemize}

\par Next we recall the description of the $H$-fixed elements in the 
distribution module $(\H_{\sigma,\lambda}^\infty)^*$. 
We first set for each $w\in \W_H\backslash \W$ 

$$V^*(\sigma, w):= (V_\sigma^*)^{w^{-1}M_Hw}\, .$$
Note that this space is one-dimensional. 
Set
$$V^*(\sigma):=\bigoplus_{w\in \W_H\backslash \W} V^*(\sigma,w) \simeq \C^{|\W_H\backslash\W|}$$
and for $w\in \W_H\backslash \W$ we denote by 
$$V^*(\sigma) \to V^*(\sigma,w), \ \ \eta\mapsto \eta_w\, .$$
the orthogonal projection. 
In the sequel we will use the 
terminology $\Re \lambda >> 0$ if 
$$(\Re \lambda -\rho)(\alpha^\vee)>0 \quad \forall \alpha\in\Sigma^+\, .$$ 
Then, for $\Re \lambda >>0$  the description

$$j(\sigma^*,-\lambda)(\eta)(g)=\begin{cases} 
a^{-\rho +\lambda} \sigma^*(m^{-1})\eta_w & \hbox{if} \ g=hwman \in HwMAN\,,   \\
0& \hbox{otherwise} \, .\end{cases}$$
defines a continuous $H$-fixed element in $\H_{\sigma^*, -\lambda}$. 
We may meromorphically continue $j(\sigma^*, \cdot)$ 
in the $\lambda$-variable and obtain, for generic values 
of $\lambda$ the identity 
$$j(\sigma^*,-\lambda)(V^*(\sigma)) =((\H_{\sigma,\lambda}^\infty)^*)^H\, .$$
For large $\lambda$ the inverse map to $j$ is given by 
$$((\H_{\sigma,\lambda}^\infty)^*)^H \ni \nu \mapsto (\nu (w))_{w\in \W_H\backslash\W} \in V^*(\sigma)\, .$$

For a smooth vector $v\in \H_{\sigma,\lambda}$ and $\eta\in V(\sigma^*)$
we obtain a smooth function on $Y=G/H$ by setting 
 
$$F_{v,\eta}(gH)= \langle \pi_{\sigma, \lambda}(g^{-1})v, 
j(\sigma^*, -\lambda)(\eta)\rangle\, .$$
The Plancherel theorem for $L^2(Y)_{\rm mc}$, see for instance \cite{BS}, 
then asserts 
the existence of a meromorphic assignment 

$$\af_\C^* \to \Gl(V(\sigma^*)), \ \ \lambda \mapsto 
C(\sigma, \lambda)$$ 
such that with $j^0(\sigma,\lambda):= j(\sigma, \lambda)\circ C(\sigma,\lambda)$
the map 

$$\Phi: \widehat{\bigoplus}_{\sigma\in\widehat{M/M_H}}  \int_{i \af_+^*}^\oplus 
\H_{\sigma,\lambda}\otimes V^*(\sigma)\ d\lambda
\to L^2(Y)_{\rm mc} $$
which for smooth vectors on the left is defined by 
$$\sum_{\sigma} (v_{\sigma,\lambda}\otimes\eta)_{\lambda}
\mapsto \left (gH\mapsto \sum_{\sigma} \int_{i\af_+^*} 
F_{v_{\sigma,\lambda}, j^0(\sigma^*, -\lambda)(\eta)}(gH) \ d\lambda\right)$$
extends to a unitary $G$-equivalence. Here $\af_+^*$ denotes a Weyl chamber 
in $\af^*$. 
 
\begin{rem} Suppose that $\W=\W_H$ (this happens in the group case). 
Then $V(\sigma^*)$ is one dimensional and we 
obtain with Remark \ref{rem=i} the following 
isomorphism: 

$$\widehat{\bigoplus}_{\sigma\in\widehat{M/M_H}}  \int_{i \af_+^*}^\oplus 
\H_{\sigma,\lambda}\otimes V(\sigma^*)\ d\lambda
\simeq \int_{i\af_+^*} L^2(\Xi)_\lambda \ d\lambda\, .$$
Hence we may view $\Phi$ as  defined on a subspace of $L^2(\Xi)$.  
\end{rem}

The inverse of the map $\Phi$ is the most continuous Fourier 
transform $\F$ (or $\F_{mc}$). For $f\in L^2(Y)_{\rm mc}\cap L^1(Y)$ 
the Fourier-transform is given by 

$$\F(f)(\sigma,\lambda,\eta)(g):=\int_Y f(y) j^0(\sigma,\lambda)(\eta)
(y^{-1}g)\ dy\, ,$$
where $\sigma\in \widehat{M/M_H}$, $\lambda\in i\af^*$ and 
$\eta\in (V^*(\sigma))^*\simeq V(\sigma^*)$. 
As a last piece of information we need to relate the Fourier-transform 
and the Radon-transform.

\subsection{The relation between Fourier and Radon transform}

Now we can determine the relation between $\cR$ and $\F$. Let us  write $\F_A^w$
for $\F_A$ on $\Xi_w= G/M_H N^w$. Let 
$f\in C_c^\infty (Y)$. 

We unwind definitions:

\begin{align*}  &\F(f)(\sigma,\lambda,\eta)(g)=\int_Y f(gy) 
j^0(\sigma,\lambda)(\eta)(y^{-1})\ dy\\
&=\sum_{w\in \W/ \W_H} \int_{ANM/w M_H w^{-1}}  f(ganmw\cdot y_o) \\
&\quad \cdot j^0(\sigma,\lambda)(\eta)(w^{-1}  m^{-1} a^{-1} n^{-1})\ da\ dn \ dm \\
&=\sum_{w\in \W/ \W_H} \int_{AM/wM_Hw^{-1}}  \cR_w (f)(gma w M_H N^w)\\ 
&\quad \cdot a^{\rho +\lambda} j^0(\sigma,\lambda)(w^{-1}  m^{-1})\ da\ \ dm \\
&=\sum_{w\in \W/ \W_H} \int_{M/w M_Hw^{-1}}  [\F_A^w \circ \cR_w] (f)(w^{-1} \lambda, g m w M_H N^w) \\
&\quad\cdot j^0(\sigma,\lambda)(w^{-1}  m^{-1}) \ dm \\
&=\sum_{w\in \W/ \W_H} \int_{M/wM_Hw^{-1}}  [\F_A^w \circ \cR_w] (f)(w^{-1} \lambda, gm w M_H N^w)\\ 
&\quad \cdot \sigma(m) j^0(\sigma,\lambda)(\eta)(w^{-1})\ \ dm
\end{align*}
Let us remark that $j^0(\sigma,\lambda)(\eta)$ is a distribution  and a 
priori the  
evaluation  $ j^0(\sigma,\lambda)(\eta)(w^{-1})$ has only meaning for 
$\Re \lambda$ 
sufficiently small. This problem is overcome by the meromorphic 
continuation of $j(\sigma,\lambda)$. This meromorphic continuation 
is in fact obtainable by an iterative procedure starting with 
$\Re \lambda$ small and larger values obtained 
by a differential operator with polynomial coefficients \cite{B}. 
This fact allows us to replace $C_c^\infty(Y)$ with the Schwartz
space $\S(Y)$ of rapidly decreasing  functions (see \cite{BS1} Sect. 12, 
and not be confused with the Harish-Chandra Schwartz space $\cC(Y)$ in Section 5 below).  
Thus we have shown:

\begin{lem} Let $f\in \S(Y)$. Then for all 
$\sigma\in \widehat{M/M_H}$, $\lambda\in i\af_+^*$
\begin{align*} & \F(f)(\sigma,\lambda,\eta)(g)= \sum_{w\in \W/ \W_H} 
\int_{M/ wM_H w^{-1}} [\F_A^w \circ \cR_w] (f)(w^{-1} \lambda, g m w M_H N^w)\\
&\quad \cdot  \sigma(m)j^0(\sigma,\lambda)(\eta)(w^{-1})\ dm \, .\end{align*}
\end{lem}

\begin{rem} The special case of $\W=\W_H$ is of particular 
interest. Then the formula from above simplifies to 

\begin{align*} & \F(f)(\sigma,\lambda,\eta)(g)= 
\int_{M/M_H} [\F_A\circ \cR] (f)(\lambda, gmM_H N)\\
&\quad \cdot  \sigma(m)j^0(\sigma,\lambda)(\eta)({\bf 1})\ dm\, .\end{align*}
\end{rem}

\begin{thm} $\cR$ restricted to $\A_{\rm mc}\cap \S(Y)$ is injective. 
\end{thm}

\begin{proof} Let $f\in \A_{\rm mc}\cap \S(Y)$. Suppose that $\cR(f)=0$. 
With Remark \ref{ran} we conclude that $\cR_w(f)=0$ for all 
$w$. Hence the lemma from above implies that $\F(f)=0$. As the 
Fourier transform is injective on $\S(Y)$, see \cite{BS1}  Cor. 12.7, 
we get that $f=0$. 
\end{proof}

\begin{rem} It is very likely that $\S(Y)\cap \A_{\rm mc}$ is 
dense in $\A_{\rm mc}$, but there does not exist a reference at 
the moment. If this were established, then the 
theorem above would imply that $\cR$ restricted to $\A_{\rm mc}$ 
is injective. 
\end{rem}

\subsection{Concluding remarks}
\subsubsection{The group case}
It is instructive to see what the results in this paper mean
for a semisimple group $G$ viewed as a symmetric space, i.e.: 
$$G\simeq G\times G/ \Delta(G)$$
with $\Delta(G)=\{ (g,g)\mid g\in G\}$ the diagonal group. 
If $P=MAN$ is a minimal parabolic of $G$ and $\oline P= MA\oline N$ 
its standard opposite (i.e. the image under the  corresponding 
Cartan involution), 
then the parameter space for the horospheres is given by 
$$\Xi= G\times G / \Delta(MA)(N\times \oline N)\, .$$

\par Our function space $\A$ are then the analytic vectors for the 
left-right regular representation of $G\times G$ on $L^1(G)$. 
For $f\in\A$ one then has 

$$ \cR(f)((g,h)\Delta(MA)(N\times \oline N))=
\int_{N\times \oline N} f(gn\oline n h^{-1})\ dn
\ d\oline n\, .$$

\subsubsection{The next steps}
Coming back to our more general situation of 
$Y=G/H$ let us consider the double fibration 
\begin{equation}\label{eq=df}
 \xymatrix { & G/ M_H \ar[dl] \ar[dr] &\\
\Xi & & Y\,.}\end{equation}
With $\cR$ comes  a dual transform $\cR^\vee$  
between appropriate function spaces $\F(\Xi)$ and $\F(Y)$ 
on $\Xi$ and $Y$: 
$$\F(\Xi)\to \F(Y);\  \cR^\vee(\phi)(gH)=\int_{H/M_H} \phi(gh \cdot M_H N)\ d(hM_H)\, .$$
For $f\in \A_{\rm mc}$ one then might ask about the existence of 
a pseudo-differential operator $D$ such that 
$$f = \cR^\vee (D\cR(f))$$
holds. For $Y=\Sl(2,\R)/\SO(1,1)$ this was considered in  \cite{GKO} 
where it was shown that such a pseudo-differential operator $D$ exists. For 
$Y$ being a group one might expect that $D$ is in fact a differential operator. 
  
\subsubsection{Radon transform on Schwartz spaces}

One might ask to what extend $\cR$ might be defined on the 
Schwartz space of $Y$. For some classes of $Y$ this seems to be possible
and we will comment on this in more detail below. Let us  
first recall the definition of the Schwartz space. 

\par One uses 
\begin{equation} \label{polar} G=KAH \end{equation}
often referred  to as the polar decomposition 
of $G$ (with respect to $H$ and $K$). Accordingly 
every $g\in G$ can be written as $g=k_g a_g h_g$ with 
$k_g\in K$ etc. It is important to notice that
$a_g$ is unique modulo $\W_H$. Therefore the prescription 

$$ \|gH\|:= |\log a_g| \qquad (g\in G)$$
is well defined for $|\cdot|$ the Killing norm on 
$\pf$. 
An alternative, and often useful,  description of $\|\cdot \|$ is as follows 

\begin{equation} \label{n2} 
\|y\|= {1\over 4} 
\left|\log \left[y\tau(y)^{-1} \theta (y \tau(y)^{-1})^{-1}\right]\right |\qquad 
(y\in Y)\, .\end{equation}
 
For $u\in \U(\gf)$ we write $L_u$ for the corresponding 
differential operator on $Y$, i.e. for $u\in\gf$ 

$$(L_u f)(y)= {d\over dt}\big|_{t=0} f(\exp(-tu)y)\, , $$
whenever $f$ is a differentiable function at $y$. 
With these preliminaries one defines the Harish-Chandra
Schwartz space of $Y$ by 
\begin{align*} \cC(Y)=\{ f\in C^\infty(Y) &\mid \forall u\in \U(\gf) 
\ \forall n\in \N\\  
& \sup_{y\in Y}  \Theta(y) (1+\|y\|)^n  |(L_uf)(y)|<\infty\}\, 
\end{align*}
where $\Theta(gH)=\phi_0 (g\tau(g)^{-1})^{-1/2}$ and $\phi_0$ 
Harish-Chandra's basic spherical function.

It is not to hard to see that $\cC(Y)$ with the obvious family 
of defining seminorms is a Fr\'echet space. Moreover 
$\cC(Y)$ is $G$-invariant and $G$ acts smoothly on it. 
We note that $\cC(Y)\subset L^2(Y)$ is a dense 
subspace. 

\par We write $BC^\infty(\Xi)$ for the space of bounded 
smooth functions on $\Xi$. 

\par In the context of defining $\cR$ on $\cC(Y)$ we 
focus we wish to discuss a basic example. 

\begin{lemma} \label{thm=R} Let $Y=\Sl(2,\R)/\SO(1,1)$, and  $f\in \cC(Y)$. Then the following assertions hold: 
\begin{enumerate}
\item The integral $\int_N f(nH)\  dn $ is absolutely 
convergent. 
\item The prescription 
$$gM_H N \mapsto \int_N f(gnH)\ dn $$ 
defines a function in $BC^\infty(\Xi)$. 
\end{enumerate}
\end{lemma}

\begin{proof} Let $A$ be the diagonal subgroup of $G$ (with positive entries) 
and $N=\begin{pmatrix} 1 & \R \\ 0& 1\end{pmatrix}$. 

\par (i) For $x\in\R$ and $n_x= \begin{pmatrix} 1 & x \\ 0& 1\end{pmatrix}$
we have to determine $a_x\in A$ such that 
$n_x \in K a_x H$. We use (\ref{n2}) and start: 

$$z_x:= n_x \tau(n_x)^{-1}= \begin{pmatrix} 1 & x  \\ 0& 1\end{pmatrix}
\cdot \begin{pmatrix} 1 & 0  \\ -x& 1\end{pmatrix}=
 \begin{pmatrix} 1-x^2 & -x  \\ x& 1\end{pmatrix}$$
and hence 
\begin{align*}y_x &:= z_x \theta(z_x)^{-1}=  
\begin{pmatrix} 1-x^2 & -x  \\ x& 1\end{pmatrix}\cdot 
\begin{pmatrix} 1-x^2 & x  \\ -x& 1\end{pmatrix}\\
&=\begin{pmatrix} (1-x^2)^2 + x^2  & *  \\ * & 1+x^2\end{pmatrix}\, .
\end{align*}
For $|x|$ large we have $\log |y_x|=|\log y_x|$. Furthermore up to 
an irrelevant constant 

\begin{align*} |y_x|&=[\operatorname{tr} (y_x y_x)]^{1\over 2}\geq {1\over 2}
[(1-x^2)^2 +x^2 + 1+x^2] \\
&\geq {1\over 2}[x^4 +1]\end{align*}
Therefore, for $|x|$ large  

$$\|n_x\|\geq  {1\over 4} \log (x^4/2 +1/2)$$

From Harish-Chandra's basic estimates of $\phi_0$ and our computation 
of $z_x$ we further get that 

$$\Theta(n_x)\geq |x|\, .$$
Therefore for $f\in \cC(Y)$ we obtain that 
$x\mapsto |f(n_x H)|$ 
growths slower than ${1\over |x|\cdot |\log x|^N}$ for any fixed 
$N>0$ and $|x|$ large. This shows (i). 

\par\noindent  (ii) Let $f\in \cC(Y)$ and set $F:=\cR(f)$. From the proof 
of (i) we know that $F$ is smooth. 
It remains to see that $F$ is bounded. 
{}From $G=KAH$ we deduce that it  is enough
to show that $F|_A$ is bounded. 
We do this by direct computation. 
For $t>0$ we set 
$$a_t=\begin{pmatrix} t & 0 \\ 0 & 1/t\end{pmatrix}\, .$$
Then 
$$a_t n_x = \begin{pmatrix} t & tx  \\ 0 & 1/t\end{pmatrix}$$
and thus 
\begin{align*} z_{t,x}&:=a_t n_x  \tau(a_t n_x)^{-1}= 
 \begin{pmatrix} t & tx  \\ 0 & 1/t\end{pmatrix}\cdot 
 \begin{pmatrix} t & 0  \\ -tx & 1/t\end{pmatrix}\\
&= \begin{pmatrix} t^2(1-x^2) & -x  \\ x & 1/t^2\end{pmatrix}
\, .\end{align*}
With that we get 
$$y_{t,x}= z_{t,x} \theta(z_{t,x})^{-1}= 
\begin{pmatrix} t^4(1-x^2)^2 +x^2  & *  \\ *  & 1/t^4 +x^2\end{pmatrix}\,.$$
For $t\geq 1$ we  conclude that 
$$ \| a_t n_x\| \gtrsim\log \left(\begin{cases} c _1 t^4 & \hbox{for}\ |x|
\leq 1/2,  \\ c_2 t^4 x^4  - c_3 & \hbox{for}\ |x|\geq 1/2\, . 
\end{cases}\right)$$
and for $|t|<1$ one has 
$$ \| a_t n_x\| \geq  \log |x|\, .$$ 
From that we obtain  (ii). 
\end{proof}

This example is somewhat specific. One might expect that 
the Radon transform on $\cC(Y)$ converges whenever 
the real rank of $G$ and of $Y$ coincide. 
\par For groups it is not hard to show that 
$\cR(f)$ does not converge for general $f\in \cC(G)$; 
integrability of $f$ is needed.

\end{document}